\documentclass[12pt,letterpaper]{amsart}
\usepackage{epic,eepic,latexsym, amssymb, amscd, amsfonts, xypic, graphicx}
\usepackage{amssymb, paralist, xspace, graphicx, url, amscd, euscript, mathrsfs, stmaryrd}
\usepackage{hyperref}

\usepackage[all]{xy}
\usepackage[usenames]{color}
 \newlength{\baseunit}               % the basic unit length
 \newcount{\numlines}                % depth of picture (in number of lines)
\newcommand{\epf}{\qed \vspace{+10pt}}

\numberwithin{equation}{section}

\newenvironment{enumeratea} {\begin{enumerate}[\upshape (a)]} {\end{enumerate}}
\newenvironment{enumeratei} {\begin{enumerate}[\upshape (i)]} {\end{enumerate}}

\newtheorem{lem}{Lemma}[section]

\newtheorem{thm}[lem]{Theorem}

\newtheorem{prop}[lem]{Proposition}

\newtheorem{cor}[lem]{Corollary}

\theoremstyle{definition}
\newtheorem{defn}[lem]{Definition}

\newtheorem{example}[lem]{Example}

\theoremstyle{remark}
\newtheorem{remark}[lem]{Remark}

   \newcommand\cF{\mathcal{F}} \newcommand\cG{\mathcal{G}}\newcommand\cU{\mathcal{U}}\newcommand\cV{\mathcal{V}}\newcommand\cW{\mathcal{W}}\newcommand\cX{\mathcal{X}}\newcommand\cY{\mathcal{Y}}\newcommand\cZ{\mathcal{Z}}

\newcommand{\PP}{\mathbb{P}}
\newcommand{\GG}{\mathbb{G}}

\newcommand{\oh}{\mathcal{O}}

\newcommand{\cx}{\mathcal{X}}

\newcommand{\cz}{\mathcal{Z}}

\newcommand{\Spec}{\operatorname{Spec}}
\newcommand{\sSpec}{\operatorname{\mathcal{S}pec}}

\newcommand{\secretnote}[1]{}

\newcommand{\lremind}[1]{{}}
\newcommand{\im}{\operatorname{im}}

\newcommand{\comment}[1]{}

\begin{document}
\pagestyle{plain}

\title{Recasting Results in Equivariant Geometry\\{\tiny Affine Cosets, Observable Subgroups and Existence of Good Quotients}}

\date{\today}

\author{Jarod D. Alper and Robert W. Easton}
%14J20:  arithmetic ground fields, 14J29 surfaces of general type, 14D15 formal methods; deformations
%\thanks{\emph{Key words and phrases}: good moduli spaces, good quotients, existence, algebraic stacks.}

%\address{??? MSRI, Berkeley, California 94720-5070 ???}
\email{jarod@math.columbia.edu \\ easton@math.utah.edu}

\begin{abstract}
Using the language of stacks, we recast and generalize a selection of results in equivariant geometry.
\end{abstract}
\maketitle

%\tableofcontents

\vspace{-0.2in}
{\parskip=12pt % closing bracket is just before the bibliography

\section{Introduction}
When an algebraic group $G$ acts on a variety $X$, there is a precise dictionary between the $G$-equivariant geometry of $X$ and the geometry of the quotient stack $[X/G]$.  This is typical of the strong interplay between equivariant geometry and algebraic stacks. Indeed, results (as well as their proofs) in the theory of algebraic stacks are often inspired by analogous results in equivariant geometry. As the simplest stacks are quotient stacks, they are fertile testing grounds for more general results. Conversely, algebraic stacks can prove quite useful for proving results in equivariant geometry.  The purpose of the present paper is to provide some examples of this power, reproving and generalizing several theorems in equivariant geometry via the language of algebraic stacks.

%example of using stacks to prove equivariant geometric results: Keel-Mori theorem

We begin in Section \ref{section-dictionary} by summarizing the relationship between the equivariant geometry of a scheme and the geometry of its corresponding quotient stack.  In Section \ref{section-goods}, we review the classical notion of a good quotient and the more modern notion of a good moduli space, and explore the relationship between them.  As a result, we recover and generalize \cite[Thm. B]{bbs_three_theorems}:

\begin{thm} \label{theorem-B}
Let $G \to S$ be an affine, linearly reductive group scheme acting on an algebraic space $X$, and suppose $X$ admits a $G$-invariant affine morphism $f: X \to Z$ to an algebraic space.  Then there exists a good quotient $\pi: X \to Y$ with $Y$ an algebraic space.
\end{thm}
\vspace{-0.2in}
\noindent Note in particular that the normality hypothesis of \cite[Thm. B]{bbs_three_theorems} has been removed, thereby answering \cite[Question pg. 149]{bbs_three_theorems}.

In Section \ref{section-affine-cosets}, we quickly recover Matsushima's theorem (see \cite{matsushima}, \cite{bb_homogeneous}, \cite{haboush_stab} and \cite{richardson}) using tools developed in Section \ref{section-goods}:

\begin{thm} \label{theorem-matsushima}
Suppose $G \to S$ is an affine, linearly reductive group scheme and $H \subseteq G$ is a flat, finitely presented, separated subgroup scheme.  Then the following are equivalent:
\begin{enumeratei}
\item $H \to S$ is linearly reductive;
\item $G/H \to S$ is affine;
\item the functor $\cF \mapsto \text{Ind}_H^G \cF$ from $\text{QCoh}^H(S)$ to $\text{QCoh}^G(S)$ is exact.
\end{enumeratei}
\end{thm}
\noindent This is a prototypical example of the power of algebraic stacks in the the study of equivariant problems. (This theorem was also proved in \cite[Thm. 12.15]{alper_good} using the same techniques, but we include the proof again here to emphasize the convenience of the language of stacks.)

Section \ref{section-observable} focuses on properties of observable subgroup schemes.  Recall that, when working over a field,  a subgroup scheme $H \subset G$ is {\em observable} if every finite dimensional $H$-representation is a sub-$H$-representation of a finite dimensional $G$-representation (see Definition \ref{def-observable} for the general case.) We find the following characterization of such subschemes:

\begin{thm}  \label{theorem-observable}
Let $G \to S$ be a flat, finitely presented, quasi-affine group scheme and $H \subseteq G$ a flat, finitely presented, quasi-affine subgroup scheme.  The following are equivalent:
\begin{enumeratei}
\item $H$ is observable;
\item for every quasi-coherent $\oh_S[H]$-module $\cF$, the counit morphism of the adjunction, $\text{Ind}_{H}^G \cF \to \cF$, is a surjection of $\oh_S[H]$-modules;
\item $BH \to BG$ is quasi-affine;
\item $G/H \to S$ is quasi-affine.
\end{enumeratei}
If $S$ is noetherian, then the above are also equivalent to:
\begin{enumeratei} \setcounter{enumi}{3}
\item every coherent $\oh_S[H]$-module is a quotient of a coherent $\oh_S[G]$-module; and
\item for every coherent $\oh_S[H]$-module $\cF$, the counit morphism of the adjunction, $\text{Ind}_{H}^G \cF \to \cF$, is a surjection of $\oh_S[H]$-modules.
\end{enumeratei}
\end{thm}
\noindent The proof of the above theorem follows directly from the observation that a representable morphism $f: \cX \to \cY$ of algebraic stacks is quasi-affine if and only if the adjunction morphism $f^*f_*\cF \to \cF$ is surjective for every quasi-coherent $\oh_{\cX}$-module $\cF$ (see Proposition \ref{proposition-quasi-affine}).

%are there other references for observable groups?   maybe add here why they're important

Lastly, in Section \ref{section-gluing} we analyze the existence of good moduli spaces, ultimately recovering a modified version of \cite[Thm. C]{bbs_three_theorems}:

\begin{thm}  \label{theorem-C}
Let $G$ be a connected algebraic group acting on a scheme $X$, and suppose that for every pair of points $x,y\in X$, there exists a $G$-invariant open subscheme  $U_{xy} \subseteq X$ that contains $x$ and $y$ and admits a good quotient.  Then $X$ admits a good quotient.
\end{thm}
\noindent Note that here we assume the group $G$ is connected, but not necessarily reductive.  In fact, it appears the proof of \cite[Thm. C]{bbs_three_theorems} is incomplete, as the constructibility of certain subsets is never verified (see Remark \ref{remark-constructibility}).  It is in the verification of that constructibility that we need to impose the connected hypothesis, as well as why we need to work with group actions rather than more general algebraic stacks.  We expect, however, a stronger version of the above theorem (as well as of Lemma \ref{lemma-constructibility}) to hold.

\begin{remark}
It is also possible to show that an analogue of \cite[Thm. A]{bbs_three_theorems} holds using similar---but significantly more involved---methods.
\end{remark}

% It is perhaps not surprising that \cite[Thm. A]{bbs_three_theorems} is more difficult since the proof relies on a reduction theorem proved in \cite{bbs_reduction} reducing the existence of a good quotient by a reductive group to the existence of good quotients by one-dimensional tori $\GG_m \subseteq G$, which of course is intimately related to the Hilbert-Mumford criterion.

\subsection*{Notation}
We fix an arbitrary (quasi-separated) base algebraic space $S$. All schemes, algebraic spaces and algebraic stacks are assumed to have quasi-compact and separated diagonals.  We use calligraphic letters $\cX, \cY, \cZ, ...$ to denote algebraic stacks and roman letters $X, Y, Z, ...$ to denote schemes and algebraic spaces.

Given an algebraic stack $\cx$, we denote by $|\cX|$ the topological space whose points correspond to equivalence classes of $\oh_S$-field-valued points of $\cx$ (see \cite[Chap. 5]{lmb}).  Any point $x \in |\cX|$ has a residue field $k(x)$, which is the coarse moduli space of the residue gerbe $\cG_x$ in $\cX$; furthermore, there exists a representative $\Spec k \to \cX$ of $x$ with $k(x)/k$ finite (see \cite[Chap. 11]{lmb}, \cite[Thm B.2]{rydh_etale-devissage}).  A point $x \in |\cX|$ induces a morphism $s: \Spec k(x) \to S$.  We will denote by $\cX_s$ the fiber product $\cX \times_S \Spec k(x)$.

%{\color{blue} I think we should sometime avoid overusing the notation $|\cX|$.  We should write $x \in |\cX|$ but for example, when taking images write $f(\cX) \subseteq |\cY|$ instead of the more precise $f(|\cX|) \subseteq |\cY|$ or even worse $|f|(|\cX|) \subseteq |\cY|$.}

%\section{Equivariant geometry vs. quotient stacks}

\section{$G$-equivariant geometry of $X$ vs. geometry of $[X/G]$} \label{section-dictionary}

For simplicity, assume momentarily $S = \Spec k$, with $k$ an algebraically closed field.  If $G \to \Spec k$ is a finite type, affine group scheme acting on a finite type $k$-scheme $X$, then we have the following dictionary between the $G$-equivariant geometry of $X$ and the geometry of $[X/G]$: %{\color{green}References?}

\begin{center}
\begin{tabular}{c|c}
{\bf $G$-equivariant geometry of $X$}			& {\bf Geometry of $[X/G]$} \\
\hline\\
orbit of $x \in X(k)$ 						& point\ $\Spec k \to [X/G]$ \\
$G$-invariant map $X \to Z$				& morphism\ $[X/G] \to Z$ \\
$\Gamma(X, \oh_X)^G$					& $\Gamma([X/G], \oh_{[X/G]})$ \\
quasi-coherent $\oh_X[G]$-module $F$	& quasi-coherent $\oh_{[X/G]}$-module\ $\cF$\\
$\Gamma(X,F)^G$ 						& $\Gamma([X/G], \cF)$\\
$G$-linearization on $X$					& line bundle on $[X/G]$\\
geometric quotient $X \to Y$				& coarse moduli space $[X/G] \to Y$\\
good quotient $X \to Y$					& good moduli space $[X/G] \to Y$\\
%{\color{blue} $\vdots$\; add more}	& {\color{blue}$\vdots$}
\end{tabular}
\end{center}

\begin{example}  Suppose $X = \Spec k$ is a point, and write $BG = [ \Spec k / G]$. Then quasi-coherent (resp., coherent) $\oh_{BG}$-modules correspond to $G$-representations (resp., of finite dimension).
\end{example}

\begin{remark}
\begin{enumeratei}
\item Technically, geometric quotients and coarse moduli spaces do not precisely agree under this dictionary. However, they do agree in the case of proper group actions.
\item The dictionary between good quotients and good moduli spaces is made precise in Section \ref{section-goods}.
\end{enumeratei}
\end{remark}

Returning now to the general case, suppose $G \to S$ is a flat, separated and quasi-compact group scheme over an arbitrary base $S$.  As usual, denote $BG = [S/G]$. Then a quasi-coherent (resp., coherent) $\oh_{BG}$-module corresponds to a $\oh_S[G]$-module (resp., of finite type).

A morphism $H \to G$ of flat, separated and quasi-compact group schemes induces a morphism $f: BH \to BG$ of algebraic stacks.  If, in addition, $H \hookrightarrow G$ is a subgroup scheme, then the diagram below is cartesian:
$$\xymatrix{
G/H \ar[r] \ar[d]		&S \ar[d] \\
BH  \ar[r]^f			& BG.
}$$
Using descent theory, we can therefore relate properties of the morphism $f: BH \to BG$ to properties of the quotient $G/H \to S$.

If $\cG$ is an $\oh_S[G]$-module, then $f^* \cG$ is the $\oh_S[H]$-module with the same underlying $\oh_S$-module as $\cG$, but with $H$-action induced from the morphism $H \to G$.
If $\cF$ is an $\oh_S[H]$-module, then $\text{Ind}_H^G \cF:= f_* \cF$ is the induced $\oh_S[G]$-module.  There is a natural morphism  $\text{Ind}_{H}^G \cF \to \cF$, corresponding to the adjunction morphism $f^*f_* \cF \to \cF$.

\section{Good quotients vs. good moduli spaces} \label{section-goods}

\subsection*{Good quotients}
Although the notion of a \emph{good quotient} was first explicitly written down by Seshadri in \cite[Definition 1.5]{seshadri_results}, the idea was already implicit in Mumford's theory of quotients by reductive groups (see \cite{git}).

\begin{defn} % ( \cite[Definition 1.5]{seshadri_results} )
Suppose $G \to S$ is an affine group scheme acting on an algebraic space $X$.  We say a morphism $\pi: X \to Y$ to an algebraic space is a \emph{good quotient} if the following hold:
\begin{enumeratei}
\item $\pi$ is surjective, affine and $G$-invariant;
\item $\oh_Y \to \pi_* (\oh_X)^G$ is an isomorphism; and
\item for every morphism $Y' \to Y$ with base change $\pi': X \times_Y Y' \to Y'$, the following hold:
\begin{enumeratea}
\item for each closed $G$-invariant subspace $Z'\subseteq X \times_{Y} Y'$, the image $\pi'(Z') \subseteq Y'$ is closed; and
\item for each pair of closed $G$-invariant disjoint subspaces $Z'_1, Z'_2 \subseteq X \times_Y Y'$, the images $\pi'(Z'_1),\pi'(Z'_2)$ are disjoint.
\end{enumeratea}
\end{enumeratei}
\end{defn}

\begin{remark}
\begin{enumeratei}
\item When $G \to S$ is linearly reductive, for $\pi: X \to Y$ to be a good quotient, it is sufficient to require only that $\pi$ is affine and $G$-invariant, and $\oh_Y \to \pi_*(\oh_X)^G$ be an isomorphism (see Corollary \ref{corollary-good-good}).
\item As in other papers (see \cite{bbs_three_theorems}, \cite{bb_book} and \cite[Tag 04AC]{stacks_project}), we have removed the assumption in \cite[Definition 1.5]{seshadri_results} that $X$ and $Y$ are schemes.  We note that the argument of \cite[Remark 0.5]{git} implies only that a good quotient is universal for $G$-equivariant maps to \emph{schemes}.  However, \cite[Theorem 6.6]{alper_good} implies that good quotients are universal for $G$-equivariant maps to algebraic spaces when $X$ is noetherian; see Corollary \ref{corollary-universal}.

\item As in \cite{bb_book} and \cite[Tag 04AC]{stacks_project}, we require that property (iii) holds for arbitrary base change.  We note, however, that it is equivalent to restrict to flat morphisms $Y' \to Y$.
\end{enumeratei}
\end{remark}

\subsection*{Linearly reductive groups}
We recall the definition of a {\em linearly reductive} group scheme. Let $\text{QCoh(S)}$ denote the category of quasi-coherent $\oh_S$-modules, and $\text{QCoh}^G(S)$ denote the category of quasi-coherent $\oh_S[G]$-modules.

\begin{defn}  A flat, finitely presented, affine group scheme $G \to S$ is \emph{linearly reductive} if the functor $\cF \mapsto \cF^G$ from $\text{QCoh}^G(S)$ to $\text{QCoh}(S)$ is exact.
\end{defn}

\begin{remark}  It is not essential to assume $G \to S$ is affine (cf. \cite[Chapter 12]{alper_good}). It is only made here to simplify the discussion.
\end{remark}

The following is well known (see, for instance, \cite[Proposition 12.6]{alper_good}):
\begin{prop}  \label{lin_red_equiv_k}
Let $G \to \Spec k$ be a finite type, affine group scheme, with $k$ a field.  The following are equivalent:
\begin{enumeratei}
\item $G$ is linearly reductive;
\item the functor $V \mapsto V^G$ from $G$-representations to vector spaces is exact;
\item the functor $V \mapsto V^G$ from finite dimensional $G$-representations to vector spaces is exact;
\item every $G$-representation is completely reducible;
\item every finite dimensional $G$-representation is completely reducible; and
\item for every finite dimensional $G$-representation $V$ and nonzero $v \in V^G$, there exists $F \in (V^{\vee})^G$ with $F(v) \ne 0$.
\end{enumeratei}
\end{prop}

\subsection*{Cohomologically affine morphisms}

For an algebraic stack $\cx$, let $\text{QCoh}(\cx)$ denote the category of quasi-coherent $\oh_{\cx}$-modules.

\begin{defn} (\cite[Definition 3.1]{alper_good}) A morphism $f: \cX \to \cY$ of algebraic stacks is \emph{cohomologically affine} if it is quasi-compact and the push-forward functor $\cF \mapsto f_* \cF$ from $\text{QCoh}(\cX)$ to $\text{QCoh}(\cY)$ is exact.
\end{defn}

\begin{remark} Cohomological affineness possesses several nice properties, such as:
\begin{enumeratei}
\item A flat, finitely presented, affine group scheme $G \to S$ is linearly reductive if and only if $BG \to S$ is cohomologically affine.

\item If $f: \cX \to \cY$ is a representable morphism of algebraic stacks and $\cY$ has quasi-affine diagonal, then $f$ is cohomologically affine if and only if $f$ is affine.

\item Cohomologically affine morphisms are stable under composition and are local on the target under faithfully flat morphisms.

\item If $f: \cX \to \cY$ is a cohomologically affine morphism of algebraic stacks and $\cY$ has quasi-affine diagonal, then for any morphism $\cY' \to \cY$ of algebraic stacks, the base change $\cX \times_{\cY} \cY' \to \cY'$ is cohomologically affine.
\end{enumeratei}
See \cite[Chapter 3]{alper_good} for additional details.
\end{remark}

We note the following interesting and well-known consequence of these properties, which stresses the necessity of quasi-affineness of the diagonal in (iv) above:

\begin{prop}
If $G \to \Spec A$ is a finite type, quasi-affine group scheme over an Artin ring $A$, then $G \to \Spec A$ is affine.
\end{prop}

\begin{proof}
We may assume $A=k$ is a field, since a scheme is affine if and only if its reduction is affine.  Observe that if $G \to \Spec k$ is any finite group scheme, then $\pi: \Spec k \to BG$ is cohomologically affine.  Indeed, the push-forward by $\pi$ of a $k$-vector space $V$ corresponds to the $G$-representation $V \otimes_k \Gamma(G, \oh_G)$, where $\Gamma(G, \oh_G)$ is the left regular representation of $G$, and this functor is clearly exact.  If $G \to \Spec k$ is quasi-affine, then $BG \to \Spec k$ has quasi-affine diagonal.  Thus the base change $G \cong \Spec k \times_{BG} \Spec k \to \Spec k$ is cohomologically affine, and hence affine.
\end{proof}

\subsection*{Good moduli spaces}

\begin{defn} (\cite[Definition 4.1]{alper_good}) \label{definition-good}
 Let $\cX$ be an algebraic stack.  We say a quasi-compact morphism $\phi: \cX \to Y$, where $Y$ is an algebraic space, is a \emph{good moduli space} if:
\begin{enumeratei}
\item $\phi$ is cohomologically affine; and
\item $\oh_Y \to \phi_* \oh_{\cX}$ is an isomorphism.
\end{enumeratei}
\end{defn}

We summarize the basic properties of good moduli spaces:

\begin{prop} (\cite[Theorems 4.16 and 6.6]{alper_good},\cite[Theorem 6.3.3]{alper_adequate}) \label{properties-good}
Let $\cX$ be an algebraic stack and $\phi: \cX \to Y$ a good moduli space.  Then:
\begin{enumeratei}
\item
	$\phi$ is surjective, universally closed and universally submersive;
\item
 	if $Z_1, Z_2$ are closed substacks of $\cX$, then
$$  \im Z_1 \cap \im Z_2 = \im (Z_1 \cap Z_2), $$
where the intersections and images are scheme-theoretic;
\item
for each algebraically closed $\oh_S$-field $k$, there is an equivalence relation defined on the set of isomorphism classes of $k$-valued points $[\cX(k)]$, given by $x_1 \sim x_2 \in [\cX(k)]$ when $\overline{ \{x_1\}} \cap \overline{ \{x_2\} } \ne \emptyset$ in $|\cX \times_S k|$, which induces a bijective map $[\cX(k)]/\kern-4pt\sim \, \, \to Y(k)$;
\item if $\cX$ is locally noetherian, then
	$\phi$ is universal for maps to algebraic spaces (that is, for any morphism to an algebraic space $\psi: \cX \to Z$, there exists a unique map $\xi: Y \to Z$ such that $\xi \circ \phi = \psi$);
\item
	if $\cX$ is locally noetherian, then $Y$ is locally noetherian and $\phi_*$ preserves coherence;
\item
	if $\cX$ is finite type over a noetherian scheme $S$, then $Y$ is finite type over $S$; and
\item
	if $Y' \to Y$ is any morphism of algebraic spaces, then the base change $\cX \times_{Y'} Y \to Y'$ is a good moduli space.
\end{enumeratei}
\end{prop}

\begin{remark}  The property of being a good moduli space also descends under faithfully flat morphisms.  See \cite[Chapter 4]{alper_good} for further properties and a systematic development of the theory.  We emphasize that other than (iv) and (vi), the proofs of the above properties are quite elementary.
\end{remark}

\subsection*{Relationship between good quotients and good moduli spaces}
We begin with:

\begin{lem}
Let $G \to S$ be an affine, linearly reductive group scheme acting on an algebraic space $X$, and let $X \to Y$ be a $G$-invariant morphism.  Then the corresponding morphism $[X/G] \to Y$ is cohomologically affine if and only if $X \to Y$ is affine.
\end{lem}

\begin{proof}
Suppose $[X/G] \to Y$ is cohomologically affine. Since $X \to [X/G]$ is affine (as $G \to S$ is affine), the composition $X \to [X/G] \to Y$ is cohomologically affine, and therefore affine.  Conversely, if $X \to Y$ is affine, consider the trivial action of $G$ on $Y$.  Since $G \to S$ is linearly reductive, $BG \to S$ is cohomologically affine.  Since $\phi$ is the composition of the affine morphism $[X/G] \to [Y/G]$ and the cohomologically affine morphism $[Y/G] \to Y$ (the base change of $BG \to S$ by $Y \to S$), $\phi$ is cohomologically affine.
\end{proof}

\begin{prop} \label{proposition-good-good}
Let $G \to S$ be an affine, linearly reductive group scheme acting on an algebraic space $X$.  Let $\pi: X \to Y$ be a $G$-invariant morphism.   Then $\phi: [X/G] \to Y$ is a good moduli space if and only if:
\begin{enumeratei}
\item $\pi$ is affine; and
\item $\oh_Y \to  \pi_*(\oh_X)^G$ is an isomorphism.
\end{enumeratei}
\end{prop}

\begin{proof}  Since there is a canonical isomorphism $\phi_* \oh_{[X/G]}\cong \pi_*(\oh_X)^G$, condition (ii) above is equivalent to condition (ii) of Definition \ref{definition-good}.  The previous lemma shows that condition (i) is equivalent to condition (i) of Definition \ref{definition-good}.
\end{proof}

In particular, we can reinterpret the definition of a good quotient in the case of an action by an affine, linearly reductive group scheme.

\begin{cor}\label{corollary-good-good}  If $G \to S$ is an affine, linearly reductive group scheme acting on an algebraic space $X$, then $\pi: X \to Y$ is a good quotient if and only if:
\begin{enumeratei}
\item $\pi$ is an affine $G$-invariant morphism; and
\item $\oh_Y \to \pi_* (\oh_X)^G$ is an isomorphism.
\end{enumeratei}
\end{cor}

\begin{proof} This follows from Propositions \ref{properties-good} and \ref{proposition-good-good}.
\end{proof}
\noindent Of course, one can show the above corollary directly without recourse to the theory of stacks and good moduli spaces.  However, we note that Proposition \ref{properties-good}(iv) implies the uniqueness of good quotients in the category of algebraic spaces.  We know of no direct proof of this result, so the language of stacks and good moduli spaces seems quite advantageous in this case.

\begin{cor}  \label{corollary-universal}
Let $G \to S$ be an affine, linearly reductive group scheme acting on a noetherian algebraic space $X$.  Suppose $\pi: X \to Y$ is a good quotient.  Then for any $G$-invariant morphism $\varphi: X \to Z$ to an algebraic space $Z$, there is a unique morphism $\Psi: Y \to Z$ such that $\varphi = \Psi \circ \pi$. \epf
\end{cor}

We can also immediately deduce:
\begin{prop}  \label{proposition-generalization-B}
Suppose an algebraic stack $\cx$ admits a cohomologically affine morphism $f: \cX \to Z$ to an algebraic space $Z$.  Then there exists a good moduli space $\phi: \cX \to Y$.
\end{prop}

\begin{proof}
Let $Y = \sSpec_Z f_* \oh_{\cX}$ and $\phi: \cX \to Y$ be the canonical morphism (hence $\oh_Y \to \phi_* \oh_{\cX}$ is trivially an isomorphism).  Consider the 2-cartesian diagram, in which the top composition is $\phi$:
$$\xymatrix{
				& \cX \ar[r]^{(id,f)} \ar[dl]	& \cX \times_{Z} Y \ar[r]^{p_2}	 \ar[dl] \ar[dr] & Y \ar[dr] \\
Y \ar[r]^{\Delta}	& Y \times_{Z} Y	&	&  \cX \ar[r]				& Z
}.$$
Since $Y \to Z$ is affine, so is $Y \to Y \times_Z Y$ and hence $\cX \to \cX \times_Z Y$.  Since $\cX \to Z$ is cohomologically affine and $Z$ has quasi-affine diagonal (as it is an algebraic space), $\cX \times_Z Y \to Y$ is cohomologically affine.  Therefore, $\phi$ is cohomologically affine.
\end{proof}

We can now prove Theorem \ref{theorem-B}:

\noindent {\bf Theorem 1.1.} {\em
Let $G \to S$ be an affine, linearly reductive group scheme acting on an algebraic space $X$, and suppose $X$ admits a $G$-invariant affine morphism $f: X \to Z$ to an algebraic space.  Then there exists a good quotient $\pi: X \to Y$ with $Y$ an algebraic space.}

\begin{proof} If $Y = \sSpec_Z f_* (\oh_X)^G$, then the induced morphism $\pi: X \to Y$ is a good quotient by Proposition \ref{proposition-generalization-B}, Proposition \ref{proposition-good-good} and Corollary \ref{corollary-good-good}.
\end{proof}

\section{Affine cosets}\label{section-affine-cosets}

We are now in the position to recover Matsushima's theorem:

\noindent {\bf Theorem 1.2.} {\em
Suppose $G \to S$ is an affine, linearly reductive group scheme and $H \subseteq G$ is a flat, finitely presented, separated subgroup scheme.  Then the following are equivalent:
\begin{enumeratei}
\item $H \to S$ is linearly reductive;
\item $G/H \to S$ is affine;
\item the functor $\cF \mapsto \text{Ind}_H^G \cF$ from $\text{QCoh}^H(S)$ to $\text{QCoh}^G(S)$ is exact.
\end{enumeratei}
}

\begin{proof}
First note that, since $G \to S$ is linearly reductive, it follows that $BG \to S$ is cohomologically affine. Suppose $H\to S$ is linearly reductive, and so $BH\to S$ is cohomologically affine.  Since $G$ is affine, $G/H \to BH$ is also affine, and hence the composition $G/H \to BH \to S$ is cohomologically affine. By Serre's criterion, $G/H\to S$ is therefore affine.  Conversely, suppose $G/H\to S$ is affine.  Consider the 2-cartesian square
\[
\xymatrix{
G/H \ar[r] \ar[d]		& S \ar[d]\\
BH \ar[r]			& BG.
}
\]
By descent, the morphism $BH \to BG$ is affine.  Therefore the composition $BH \to BG \to S$ is cohomologically affine, and so $H \to S$ is linearly reductive.  This proves the equivalence of (i) and (ii).

Condition (iii) is a direct translation of the condition of affineness for the morphism $BH \to BG$, and is thus equivalent by descent and Serre's criterion to (ii).
%Lastly, we prove (iii) is equivalent to (iv).  Supposing (iii), it follows from (ii) that $f: BH \to BG$ is affine, and hence the counit morphism of the adjunction, $f^*f_* \cF \to \cF$, is surjective for any $\oh_S[H]$-module $\cF$ (see for instance Prop. \ref{proposition-quasi-affine}).   Taking $\cG = f_* \cF$, we then have that $f^* \cG \to \cF$ is surjective and $\cG^G \cong \cF^H$.  Conversely...{\color{green}MISSING}.
\end{proof}

The following result was used by Bia\l ynicki-Birula (\cite{bb_homogeneous}) and Richardson (\cite{richardson} to prove Matsushima's theorem.  A proof also appeared in \cite[p. 85]{luna}.  %Richardson asks in \cite[pg. 40]{richardson} if there is an elementary proof of the result.
The language of stacks provides a quick proof relying essentially only on descent for affine morphism.

\begin{prop}  ({\rm cf.} \cite[Lemma 1]{bb_homogeneous}) \label{proposition-cosets}
Let $G \to S$ be a flat, finitely presented, separated group scheme, and $H_2 \subseteq H_1$ be an inclusion of flat, finitely presented, separated subgroup schemes of $G$, with $H_1$ quasi-affine over $S$. Then $H_1/H_2 \to S$ is affine if and only if $G/H_2 \to G/H_1$ is affine.
\end{prop}

\begin{proof}
The 2-cartesian diagrams
$$\begin{aligned}
\xymatrix{
H_1/H_2 \ar[r] \ar[d]		& S \ar[d] \\
BH_2 \ar[r]			& BH_1
}	& \qquad
\xymatrix{
G/H_2 \ar[r] \ar[d]	& G/H_1 \ar[r] \ar[d]		& S \ar[d] \\
BH_2 \ar[r]		& BH_1 \ar[r]			& BG
}
\end{aligned}$$
and descent theory immediately imply the result.
\end{proof}

\begin{cor} ({\rm cf.} \cite[Corollary 1]{bb_homogeneous})  Let $H_2 \subseteq H_1 \subseteq G$ be inclusions of group schemes as in Proposition \ref{proposition-cosets}.  If $H_1/H_2$ and $G/H_1$ are affine over $S$, then so is $G/H_2$. \epf
\end{cor}

\section{Observable subgroups}\label{section-observable}

For a subgroup $H \subseteq G$, Matsushima's theorem provides a relationship between the reductiveness of $H$ and the affineness of $G/H$ (see Theorem \ref{theorem-matsushima}).  In this section, we analyze the relationship between the observability of $H$ (i.e., every $H$-representation is a subrepresentation of a $G$-representation) and the quasi-affineness of $G/H$.  Our approach is in the same sprit as the proof of Theorem \ref{theorem-matsushima}; we interpret the quasi-affineness of $G/H$ in terms of functorial properties of $BH \to BG$. We first prove the following characterization of quasi-affine morphisms generalizing \cite[II.5.1.2 and IV.5.1.2]{ega}:

\begin{prop}  \label{proposition-quasi-affine}
Let $f: \cX \to \cY$ be a quasi-compact (and quasi-separated) morphism of algebraic stacks, with $\cY$ having quasi-affine diagonal.  Suppose $f$ is either strongly representable, or finite presentation and representable.  Then the following are equivalent:
\begin{enumeratei}
\item $f$ is quasi-affine;
\item for any quasi-coherent $\oh_{\cX}$-module $\cF$, the adjunction morphism $\phi^* \phi_* \cF \to \cF$ is surjective; and
\item for any quasi-coherent $\oh_{\cX}$-module $\cF$, there exists a quasi-coherent $\oh_{\cY}$-module $\cG$ and a surjection $f^* \cG \to \cF$;
\end{enumeratei}
If, in addition, $\cX$ and $\cY$ are locally noetherian, then the above conditions are also equivalent to:
\begin{enumeratei} \setcounter{enumi}{3}
\item for any coherent $\oh_{\cX}$-module $\cF$, the adjunction morphism $\phi^* \phi_* \cF \to \cF$ is surjective; and
\item for any coherent $\oh_{\cX}$-module $\cF$, there exists a coherent $\oh_{\cY}$-module $\cG$ and a surjection $f^* \cG \to \cF$.
\end{enumeratei}
\end{prop}

\begin{proof}
First, it is clear (ii) implies (iii).  To see (iii) implies (ii), suppose $\cG$ is a quasi-coherent $\oh_{\cY}$-module and $f^* \cG \to \cF$ is a surjection. The counit of the adjunction then gives a factorization $f^*\cG \to f^* f_* \cF \to \cF$, and hence $f^* f_* \cF \to \cF$ is also surjective.

We next show (i) is equivalent to (ii).  Note that the property of being quasi-affine is stable under composition and base change and descends under faithfully flat morphisms on the target.   Also, it is easy to see property (ii) for a quasi-compact (and quasi-separated) representable morphism of algebraic stacks is also stable under composition and descends under faithfully flat morphisms on the target.  Furthermore, if the target has quasi-affine diagonal, then property (ii) is stable under arbitrary base change.  Therefore, we can immediately reduce to the case where $f: X \to Y$ is a quasi-compact (and quasi-separated) morphism of algebraic spaces with $Y$ an affine scheme.

As open immersions and affine morphisms satisfy property (ii), quasi-affine morphisms also do.  Thus, (i) implies (ii).  For the converse, we may suppose that $f: X \to Y = \Spec \Gamma(X,\oh_X)$.  If $f$ is strongly representable, then $X$ is a scheme, so \cite[II.5.1.2 and IV.5.1.2]{ega} implies that $X$ is quasi-affine.  Now suppose that $f$ is finite type.  We first claim that for every quasi-coherent $\oh_X$-module $\cF$, the adjunction morphism $f^* f_* \cF \to \cF$ is an isomorphism.  Indeed, there exists a surjection $\gamma: \bigoplus_i \oh_Y \to f_* \cF$ which induces a surjection $\alpha: \bigoplus_i \oh_X \to f^* f_* \cF \to \cF$ (since $\cF$ is globally generated), with the property that $f_* \alpha = \gamma$.  By repeating, there is a presentation
$$\bigoplus_{j \in J} \oh_X \xrightarrow{\beta} \bigoplus_{i \in I} \oh_X \xrightarrow{\alpha} \cF \to 0$$
which remains right exact after applying $f_*$.  Therefore, we have a commutative diagram
$$\xymatrix{
f^* f_* (\bigoplus_{j \in J} \oh_X) \ar[r] \ar[d]	& f^* f_* ( \bigoplus_{i \in I} \oh_X ) \ar[r] \ar[d]	 & f^*f_* \cF \ar[r] \ar[d]	& 0 \\
\bigoplus_{j \in J} \oh_X \ar[r]		& \bigoplus_{i \in I} \oh_X  \ar[r]		& \cF \ar[r] & 0
}$$
with exact rows.  Since the left two vertical arrows are isomorphisms, so is $f^* f_* \cF \to \cF$.

We next show that $f$ is separated.  By \cite[Theorem B]{rydh_noetherian}, there exists a finite surjective morphism $U \to X$ from a scheme $U$.  Since the adjunction morphism $f^* f_* \cF \to \cF$ is an isomorphism for quasi-coherent $\oh_{X}$-modules, the diagram
$$\xymatrix{
U \ar[r] \ar[d]	& \Spec \Gamma(U, \oh_U) \ar[d] \\
X \ar[r]		& \Spec \Gamma(X,\oh_X)
}$$
is cartesian.  Since $U$ is a scheme and $U \to \Spec \Gamma(U, \oh_U)$ satisfies property (ii), $U \to \Spec \Gamma(U, \oh_U)$ is an open immersion.  In particular, $U$ is separated and it follows that $X$ is separated.

We show now that $f$ is a monomorphism. Let $h_1, h_2: T \to X$ be morphisms from an affine scheme $T$ such that $g = h_1 \circ f = h_2 \circ f$:
$$\xymatrix{
T\ar@<.5ex>[r]^{h_1} \ar@<-.5ex>[r]_{h_2} \ar[rd]_g	& X \ar[d]^f \\
								 		& \Spec \Gamma(X,\oh_X).
}$$
Since $X$ has affine diagonal, the morphisms $h_i$ are affine and thus determined by the quasi-coherent $\oh_X$-algebras $(h_i)_* \oh_T$.  Since adjunction is an isomorphism, we have $(h_i)_* \oh_T = f^* f_* (h_i)_* \oh_T = f^* g_* \oh_T$, so $h_1 = h_2$.

Finally, using the assumption that $f: X \to \Spec \Gamma(X, \oh_X)$ is finite presentation, we have that $f$ is a quasi-finite, separated representable morphism of algebraic spaces, hence by Zariski's main theorem (\cite[Thm. 16.5]{lmb}), $f$ is quasi-affine.

The equivalences with (iv) and (v) follow by standard direct limit methods using \cite[Theorem 15.4]{lmb}.
\end{proof}

%{\color{blue} I still would like to drop the finite type hypothesis despite the remark below.  His argument is really neat because it uses this cool fact that the property of being an AF-scheme descends under finite morphisms.  But I also like the above argument because it is a direct manipulation of properties of quasi-coherent sheaves.}

\begin{remark}
\begin{enumeratei}
\item By a different method, Philipp Gross has recently shown the above proposition holds more generally for any quasi-compact (and quasi-separated) representable morphism $f: \cX \to \cY$ of algebraic stacks, with $\cY$ having quasi-affine diagonal.

\item The assumption that $\cY$ have quasi-affine diagonal is necessary for (ii) to imply (i).  For example, if $G \to \Spec k$ is any finite type group scheme, then $\Spec k \to BG$ satisfies (ii).  Indeed, for any $k$-vector space $V$, the adjunction map corresponds to the $V \otimes_k \Gamma(G, \oh_G) \to V$, which is surjective.  However, if $G \to \Spec k$ is an abelian variety, then $\Spec k \to G$ is not quasi-affine, and hence does not satisfy (i).
\end{enumeratei}
\end{remark}

\begin{defn}\label{def-observable}  Let $G \to S$ be a flat, finitely presented, quasi-affine group scheme.  A flat, finitely presented, quasi-affine subgroup scheme $H \subseteq G$ is \emph{observable} if every quasi-coherent $\oh_S[H]$-module is a quotient of a quasi-coherent $\oh_S[G]$-module.
\end{defn}

\begin{remark} If $S = \Spec k$, this is easily seen to be equivalent to the definition in \cite{bbhm}:  a subgroup scheme $H \subset G$ is observable if every finite dimensional $H$-representation is a sub-$H$-representation of a finite dimensional $G$-representation.
\end{remark}

We can now prove:

\noindent {\bf Theorem 1.3.} {\em
Let $G \to S$ be a flat, finitely presented, quasi-affine group scheme and $H \subseteq G$ a flat, finitely presented, quasi-affine subgroup scheme.  The following are equivalent:
\begin{enumeratei}
\item $H$ is observable;
\item for every quasi-coherent $\oh_S[H]$-module $\cF$, the counit morphism of the adjunction, $\text{Ind}_{H}^G \cF \to \cF$, is a surjection of $\oh_S[H]$-modules;
\item $BH \to BG$ is quasi-affine; and
\item $G/H \to S$ is quasi-affine.
\end{enumeratei}
If $S$ is noetherian, then the above are also equivalent to:
\begin{enumeratei} \setcounter{enumi}{3}
\item every coherent $\oh_S[H]$-module is a quotient of a coherent $\oh_S[G]$-module; and
\item for every coherent $\oh_S[H]$-module $\cF$, the counit morphism of the adjunction, $\text{Ind}_{H}^G \cF \to \cF$, is a surjection of $\oh_S[H]$-modules.
\end{enumeratei}
}

\begin{proof}  The equivalences follow from the definitions and Proposition \ref{proposition-quasi-affine}.  We only add that since $S \to BG$ is faithfully flat and finitely presented and $G/H \cong BH \times_{BG} S$, descent implies $BH \to BG$ is quasi-affine if and only if $G/H \to S$ is quasi-affine.
\end{proof}

\section{Existence of good moduli spaces}\label{section-gluing}

Given an algebraic stack $\cx$, we would like local conditions that guarantee the existence of a good moduli space.  Recall first the following definition:

\begin{defn}(\cite[Def. 6.1]{alper_good}) If $\pi:\cx \to Y$ is a good moduli space, an open substack $\mathcal{U}\subseteq \cx$ is {\em saturated for $\pi$} if $\pi^{-1}(\pi(\mathcal{U}))=\mathcal{U}$.
\end{defn}

Saturated substacks have the following nice property (see \cite[Rmk. 6.2]{alper_good}): if an open substack $\mathcal{U}\subseteq \cx$ is saturated for $\pi$, then $\pi(\mathcal{U})$ is open and $\pi|_{\mathcal{U}}:\mathcal{U}\to \pi(\mathcal{U})$ is a good moduli space.  Moreover, we have the following proposition:

\begin{lem} (\cite[Prop. 7.9]{alper_good}) \label{lemma-glue}
Suppose $\cx$ is a noetherian algebraic stack containing open substacks $\{\mathcal{U}_i\}_{i\in I}$ such that for each $i$ there exists a good moduli space $\pi_i:\mathcal{U}_i\to Y_i$, with $Y_i$ a scheme (resp., an algebraic space).  Let $\mathcal{U}=\bigcup \mathcal{U}_i$.  Then the following are equivalent:
\begin{enumeratei}
\item $\mathcal{U}_i\cap \mathcal{U}_j$ is saturated for $\pi_i$ for every $i,j$; and
\item there exists a good moduli space $\pi:\mathcal{U}\to Y$ with $Y$ a scheme (resp., an algebraic space), and algebraic subspaces $\widetilde{Y}_i\subseteq Y$ with $\widetilde{Y}_i\cong Y_i$ and $\pi^{-1}(\widetilde{Y}_i)=\mathcal{U}_i$.
\end{enumeratei}
\end{lem}

It is useful to have an intrinsic definition of saturated that does not refer to a good moduli space, and to use this definition to find conditions that guarantee the existence of open covers by saturated substacks.  Combined with the previous proposition, this would enable us to give local conditions guaranteeing local good moduli spaces always glue.  We first make the following definition:

\begin{defn} Suppose $\cX$ is an algebraic stack over a scheme $S$ and $\cZ \subseteq \cX$ is a closed substack.  Define $F_{\cX}(\cZ)\subseteq |\cX|$ to be the set of those points $x\in |\cX|$ over maps $s: \Spec k(x) \to S$ such that $\overline{\{x\}} \cap \cZ_s \neq \emptyset$ in $|\cX_s|$ (where by abuse of notation we also consider $x \in |\cX_s|$ and the closure $\overline{ \{x \}} \subseteq |\cX_s|$).
\end{defn}

\begin{remark}
\begin{enumeratei}
\item While the definition of $F_{\cX}(\cZ)$ appears to depend on the base scheme $S$, it is in fact independent.

%\item Suppose $\cX$ is an Artin stack of $\Spec k$ and $x \in |\cX|$ is a point such that $k \to k(x)$ is finite.  Let $\cZ \subseteq \cX$ be a closed substack.  Then $x \in F_{\cX}(\cZ)$ if and only if $\overline{ \{ x \} } \cap \cZ \neq \emptyset$ in $|\cX|$.

%\item Let $x \in |\cX|$ be over $s: \Spec k(x) \to S$, and $k(x) \to k$ be any field extension.  Let $t: \Spec k \to \Spec k(x) \to S$.  If $\cW_s$ (resp., $\cW_t$) denotes the closure of $\{x\}$ in $\cX_s$ (resp., $\cX_t$), then $\cW_s \times_{k(x)} k \cong \cW_t$.  Therefore, $x \in F_{\cX}(\cZ)$ if and only if there is a representative $\Spec k \to \cX$ of $x$ with image $t: \Spec k \to \cX \to S$ such that $\overline{\{x\}} \cap \cZ_t \neq \emptyset$ in $|\cX_t|$.

%\item Note also that if $\cX \to \Spec k$ is an algebraic stack, $\cZ \subseteq \cX$ a closed substack, and $x: \Spec k' \to \cX$ a morphism such that $k'/k$ is algebraic, then $\overline{ \{x\} } \cap \cZ \neq \emptyset$ in $|\cX|$ if and only if $\overline{ \{x \}} \cap (\cZ \times_k k') \neq \emptyset$ in $|\cX \times_k k'|$.

\item Note that if $F_{\cX}(\cZ) \subseteq |\cX|$ is closed, then $F_{\cX}(F_{\cX}(\cZ)) = F_{\cX}(\cZ)$.
\end{enumeratei}
\end{remark}

%{\color{blue} I should check the first two remarks again but I think they are both right.. note that in order to define $F(\cZ)$ appropriately even when $S = \Spec k$ for an algebraically closed field $k$, we need to deal with such issues.}

%{\color{blue} I have changed the notation here.  Even it's perhaps more precise, I feel that the notation of representatives $(\xi, s)$ is awkward.  I'd prefer not choose representatives for points in $|\cX|$.  However, given $x \in |\cX|$, there is a representative of its image $s \in S$ given by $\Spec k(x) \to S$ (which factors through $\Spec k(s) \to S$).  I've changed the notation section as well. }

\begin{lem} \label{lemma-f-closed}
Suppose $\pi:\cx \to Y$ is a good moduli space.  Then $|\pi^{-1}(\pi(\cz))|=F_{\cX}(\cz)$ for every closed substack $\cz\subseteq \cx$.  In particular, $F_{\cX}(\cZ)$ is closed.
\end{lem}

\begin{proof}
%First note that %$|\pi^{-1}(\pi(\cz))|=\pi^{-1}(|\pi(\cz)|)=\pi^{-1}(\pi(|\cz|))$ and
%$\pi(\cz)_s\cong \pi_s(\cz_s)$ for any closed substack $\cz\subseteq \cx$ and point $s \in S$ (by \cite[Thm. 4.15]{alper_good}).
Suppose $x\in \pi^{-1}(\pi(\cz))$ and let $s: \Spec k(x) \to S$.  Then $\pi_s(\overline{\{x\}})\cap \pi_s(\cz_s)\neq \emptyset$ in $Y_s$.  By \cite[Thm. 4.15(iii)]{alper_good}, this implies $\overline{\{x\}}\cap \cz_s\neq \emptyset$ in $\cx_s$, and hence $x\in F_{\cX}(\cZ)$. Conversely, suppose $x\in F_{\cX}(\cz)$ over $s \in S$ and choose any point $z \in \overline{\{x\}}\cap \cz_s \subseteq |\cx_s|$.  Then $\overline{\{x\}} \cap\overline{\{z\}} \neq \emptyset$ in $\cx_s$, and so by \cite[Thm. 4.15(iv)]{alper_good} we have $\pi_s(x)=\pi_s(z)$. %, i.e., $x\in \pi_s^{-1}(|\pi_s(\cz_s)|)$.
It follows that $x\in \pi^{-1}(|\pi(\cz)|)$.
\end{proof}

\begin{lem} \label{lemma-saturated}
Suppose $\pi:\cx \to Y$ is a good moduli space.  Let $\cU \subseteq \cX$ be an open substack with reduced complement $\cZ$.  Then the following are equivalent:
\begin{enumeratei}
\item $\cU$ is saturated for $\pi$;
\item $F_{\cX}(\cZ) = \cZ$;
\item for every point $u\in |\mathcal{U}|$ over $s: \Spec k(u) \to S$, the closure $\overline{\{u\}} \subseteq |\cX_s|$ is contained in $|\cU_s|$; and
\item every point $u \in |\cU|$ over $s: \Spec k(u) \to S$ that is closed in $\cU_s$ is also closed in $\cX_s$.
\end{enumeratei}
\end{lem}

\begin{proof}
First note that $\pi^{-1}(\pi(\cU)) = \cU$ if and only if $\pi^{-1}(\pi(\cZ)) = \cZ$; indeed both are equivalent to the statement that for every $s: \Spec k \to S$, $u \in |\cU_s|$ and $z \in |\cZ_s|$, one has $\overline{ \{u\}} \cap \overline{ \{z \}} = \emptyset$ in $|\cX_s|$.  By Lemma \ref{lemma-f-closed}, we have that (i) through (iii) are equivalent.  It is clear (iii) implies (iv).  Conversely, let $u \in |\cU|$ over $s: \Spec k(u) \to S$ be such that $x_0 \in \overline{ \{ u \}} \cap \cZ_s \subseteq |\cX_s|$ is a closed point.  Then for any closed point $u_0 \in \overline{ \{u\}} \subseteq |\cU_s|$, we also have that $x_0 \in \overline{ \{ u_0 \}}$ by \cite[Thm 4.15]{alper_good}.  In particular, $u_0 \in |\cU_s|$ is closed but $u_0 \in |\cX_s|$ is not closed, violating (iv).
\end{proof}

\begin{remark} If $\cX$ is finite type over $S = \Spec k$ with $k$ algebraically closed, then the equivalence of (i) and (iv) simply states that an open substack $\cU \subseteq \cX$ is saturated for $\pi$ if and only if every closed point $u \in \cU(k)$ is also closed in $\cX$, i.e., the open immersion $\cU \hookrightarrow \cX$ maps closed points to closed points.
\end{remark}

By Lemma \ref{lemma-saturated}, it is reasonable to make the following definition:

\begin{defn} An open substack $\mathcal{U}\subseteq \cx$ is {\em saturated} if for every point $u\in |\mathcal{U}|$ over $s: \Spec k(u) \to S$ , the closure $\overline{\{u\}} \subseteq |\cX_s|$ is contained in $|\cU_s|$.
\end{defn}

\begin{remark} Note that $\cU \subseteq \cX$ is saturated if and only if $F_{\cX}(\cZ) = |\cZ|$ where $\cZ = \cX \setminus \cU$ is the reduced complement.  See \cite[Section 2]{alper_quotient} for more general notions of saturated and weakly saturated morphisms, as well as their properties.
\end{remark}

As a converse to Lemma \ref{lemma-f-closed}, we have the following lemma:

\begin{lem} (cf. \cite[Lem. 1]{bbs_three_theorems}) \label{lemma-saturated-opens}
Suppose the set $F_{\cX}(\cZ)\subseteq |\cX|$ is closed for every closed substack $\cz\subseteq \cx$. Then for every open substack $\cU\subseteq \cx$, there exists an open substack $\cV\subseteq \cU$ such that:
\begin{enumeratei}
\item $\cV \subseteq \cX$ is saturated; and
\item for every $u\in |\cU|$ over $s: \Spec k(u) \to S$ with $\overline{\{u\}} \subseteq \mathcal{\cX}_s$ contained in $\cU_s$, one has $u \in |\cV|$.  (In particular, any point $u \in |\cU|$ over $s: \Spec k(u) \to S$ which is closed in $\cX_s$ is contained in $\cV_s$.)
\end{enumeratei}
\end{lem}

\begin{proof}
Let $\cz$ be the reduced closed complement $\cX \setminus \cU$.   Then $\cV = \cX \setminus F_{\cX}(\cZ) \subseteq \cX$ is a saturated open substack contained in $\cU$ with the desired property.
% Then $F_{\cX}(\cz)$ is closed and contains $|\cz|$, so $|\cx|\backslash F_{\cX}(\cz)$ is open and contained in $|\mathcal{V}|$.  It is therefore of the form $|\mathcal{U}|$ for some open substack $\mathcal{U}\subseteq \mathcal{V}$. If a point $x\in |\mathcal{U}|$ is represented by a pair $(\xi,s)$, then $\overline{\{x\}}_s\cap \cz_s=\emptyset$, and it follows that for every point $x'\in \overline{\{x\}}_s$ we also have $\overline{\{x'\}}_s\cap \cz_s=\emptyset$; i.e., $x'\in |\mathcal{U}_s|$.  Therefore $\overline{\{x\}}_s\subseteq \mathcal{U}_s$, and so $\mathcal{U}$ is indeed saturated.  By construction, $x\in |\mathcal{U}|$ whenever $\overline{\{x\}}_s\subseteq \mathcal{V}_s$.
\end{proof}

Since we are interested in determining conditions guaranteeing the existence of a good moduli space, Lemmas \ref{lemma-glue}, \ref{lemma-f-closed} and \ref{lemma-saturated-opens} suggest we should establish when the sets $F_{\cX}(\cZ)$ are closed for all closed substacks $\cZ\subseteq \cX$.  We first give conditions guaranteeing the sets $F_{\cX}(\cZ)$ are constructible.

\begin{lem} \label{lemma-f-constructible}  Suppose $\cX$ be a noetherian algebraic stack such that there exists a locally quasi-finite, universally submersive morphism $f:\cW \to \cX$ from an algebraic stack $\cW$ admitting a good moduli space.  Then for every closed substack $\cZ \subseteq \cX$, the set $F_{\cX}(\cZ) \subseteq \cX$ is constructible.
\end{lem}

\begin{remark}  We note that \'etale morphisms and finite morphisms are locally quasi-finite and universally submersive.
\end{remark}

\begin{proof}
By Lemma \ref{lemma-f-closed}, $F_{\cW}(f^{-1}(\cZ)) \subseteq |\cW|$ is closed.  We claim that $f( F_{\cW}(f^{-1}(\cZ))) = F_{\cX}(\cZ)$.  The containment $\subseteq$ is clear.  Conversely, if $x \in F_{\cX}(\cZ)$ is over $s: \Spec k(x) \to S$, then there exists a specialization $x \rightsquigarrow x_0$ with $x_0 \in |\cZ_s|$.  Since $\cW_s \to \cZ_s$ is quasi-finite and submersive, there is a specialization $w \rightsquigarrow w_0$ in $|\cW_s|$ over $x \rightsquigarrow x_0$.  Furthermore, the field extension $k(x) \to k(w)$ is finite which implies that $w \in F_{\cW}(f^{-1}(\cZ))$ if and only if $\overline{\{w \}} \cap f^{-1}(\cZ) \neq \emptyset$.  Therefore, $w \in F_{\cW}(f^{-1}(\cZ))$ and $x \in f( F_{\cW}(f^{-1}(\cZ)))$.
\end{proof}

\begin{lem} \label{lemma-constructibility}
Let $\cX$ be an algebraic stack of finite type over an algebraically closed field $k$.  Suppose that:
\begin{enumeratei}
\item $\cX \cong [X/G]$, where $X$ is a scheme and $G$ is a connected algebraic group; and
\item stabilizers of closed points in $\cX$ are linearly reductive.
\end{enumeratei}
 Then for every closed substack $\cZ \subseteq \cX$, the set $F_{\cX}(\cZ) \subseteq \cX$ is constructible.
\end{lem}

\begin{proof}  If $X$ is smooth, the statement follows from \cite[Theorem 3]{alper_quotient}) and Lemma \ref{lemma-f-constructible}.  We may apply Sumihiro's theorem (\cite{sumihiro1}) to reduce the normal case to the smooth case.  The general case follows since normalization is finite and, in particular, quasi-finite and universally submersive.
\end{proof}

%{\color{blue} It would be nice to show constructibility much more generally so that we don't need to rely on a quotient stack presentation.  But I don't see how.. Do you see another way to show constructibility?}

\begin{remark}  \label{remark-constructibility}
It appears that in the proof of \cite[Lemma 2]{bbs_three_theorems}, the constructibility of $\{x \in X \mid Gx \cap Y \neq \emptyset \}$ is not verified for the action of the reductive group $G$ on the algebraic variety $X$.  It is checked the set is closed under specialization, but one needs constructibility of the set to then conclude it is closed.
\end{remark}

\begin{lem} (cf. \cite[Lem. 2b]{bbs_three_theorems}) \label{lemma-f-closed-criterion}
Suppose $\cX$ satisfies the hypotheses of Lemma \ref{lemma-constructibility}.
Suppose further that either:
\begin{enumeratei}
\item for every smooth curve $C$ over $k$ and morphism $f: C \to \cX$, the scheme-theoretic image of $f$ admits a good moduli space; or
\item for every pair of points $x,y\in |\cX|$, there exists an open substack $\cU_{xy} \subseteq \cX$ that contains $x$ and $y$ and admits a good moduli space.
\end{enumeratei}
Then for every closed substack $\cZ \subseteq \cX$, the set $F_{\cX}(\cZ) \subseteq \cX$ is closed.
\end{lem}

\begin{proof}
By Lemma \ref{lemma-f-constructible}, the sets $F_{\cX}(\cZ)$ are constructible, so it suffices to check that $F_{\cX}(\cZ)$ is closed under specialization.
First, suppose condition (i) holds. If $F_{\cX}(\cZ)$ is not closed for some closed substack $\cZ\subseteq \cX$, then there exists a smooth pointed curve $(C,p)$ and a morphism $f: C \to \cX$ such that $f(C \setminus p) \subseteq F_{\cX}(\cZ)$ but $f(p) \notin F_{\cX}(\cZ)$.  By assumption, the scheme-theoretic image $\cY \subseteq \cX$ of $f$ admits a good moduli space.  But then Lemma \ref{lemma-f-closed} implies $F_{\cX}(\cZ) \cap \cY = F_{\cY} (\cZ \cap \cY)$ is closed, a contradiction.

Now suppose condition (ii) holds.  If $F_{\cX}(\cZ)$ is not closed for some closed substack $\cZ\subseteq \cX$, then there exists a closed point $x \in \overline{F_{\cX}(\cZ)} \setminus F_{\cX}(\cZ)$ that is a specialization of $x' \in F_{\cX}(\cZ)$.  By assumption, there are finitely many points $y_1, \ldots, y_k \in |\cX|$ and open substacks $\cU_{xy_i}$ containing $x$ and $y_i$, such that $\cU_{xy_i}$ admits a good moduli space and $\bigcup_i \cU_{xy_i} = \cX$.  Note that $ F_{\cU_{xy_i}}(\cX \cap \cU_{xy_i}) \subseteq |\cU_{xy_i}|$ is closed and $F_{\cX}(\cZ) = \bigcup_i F_{\cU_{xy_i}}(\cX \cap \cU_{xy_i})$.  But $x' \in F_{\cU_{xy_i}}(\cZ \cap \cU_{xy_i})$ for some $i$, which contradicts $x \notin F_{\cX}(\cZ)$.
\end{proof}

\begin{remark} In (ii) above, if one were instead to require the weaker property that every point have a open neighborhood admitting a good moduli space, then the conclusion would no longer hold.   Consider, for example, the stack $\cX = [\PP^1 / \GG_m]$, where $\GG_m$ acts by multiplication. In this case, $F_{\cX} (\{ 0 \}) = [ (\PP^1 \setminus \{ \infty\}) / \GG_m]$ is not closed.
\end{remark}

\begin{prop}  \label{prop-good}
Suppose $\cX$ is a noetherian algebraic stack such that:
\begin{enumeratei}
\item every point $x \in |\cX|$ has an open neighborhood admitting a good moduli space; and
\item for every closed substack $\cZ \subseteq \cX$, $F_{\cX}(\cZ) \subseteq |\cX|$ is closed.
\end{enumeratei}
Then $\cX$ admits a good moduli space.
\end{prop}

\begin{proof}  For a closed point $x \in |\cX|$, let $\cU_x$ be an open neighborhood admitting a good moduli space.  By Lemma \ref{lemma-saturated-opens}, there exists an open neighborhood $\cV_x \subseteq \cU_x$ containing $x$ such that $\cV_x \subseteq \cX$ is saturated.  It follows also that $\cV_x \subseteq \cU_x$ is saturated, and so $\cV_x$ also admits a good moduli space.  By Lemma \ref{lemma-glue}, the good moduli spaces of $\cV_x$ can be glued to construct a good moduli space of $\cX$.
\end{proof}

The proof of the following theorem now follows directly from Lemma \ref{lemma-f-closed-criterion} and Proposition \ref{prop-good}.

\noindent {\bf Theorem 1.4.} {\em
Let $G$ be a connected algebraic group acting on a scheme $X$, and suppose that for every pair of points $x,y\in X$, there exists a $G$-invariant open subscheme  $U_{xy} \subseteq X$ that contains $x$ and $y$ and admits a good quotient.  Then $X$ admits a good quotient.}
\epf
}

\bibliography{references}{}
\bibliographystyle{amsalpha}

\end{document}